\documentclass[12pt]{article}
\usepackage{amssymb,amsmath,bm}

\usepackage[colorlinks=true, urlcolor=blue,linkcolor=blue, citecolor=blue]{hyperref}

\usepackage{mathrsfs,verbatim}
\usepackage{caption,cleveref }
\usepackage{graphicx, float}
\usepackage{color, mathtools, tikz, xcolor}
\usepackage{enumerate}
\usepackage{tikz}
\usetikzlibrary{shapes.geometric}
\usepackage{pdfpages}

\usepackage{caption}

\usetikzlibrary{positioning}
\usetikzlibrary{arrows}
\usepackage{subcaption}

\usetikzlibrary{calc}

\usetikzlibrary{decorations.pathreplacing}

\usepackage{appendix}
\usepackage{array}

\usepackage{multirow}
\usepackage{graphicx}
\newcolumntype{L}{>{$}l<{$}}
\newcolumntype{D}{>{\centering\arraybackslash}p{3.5cm}}
\newcolumntype{C}{>{$}c<{$}}

\newcommand{\bea}{\begin{eqnarray}}
\newcommand{\ena}{\end{eqnarray}}
\newcommand{\beas}{\begin{eqnarray*}}
\newcommand{\enas}{\end{eqnarray*}}
\newcommand{\beq}{\begin{equation}}
\newcommand{\enq}{\end{equation}}

\newcommand{\ignore}[1]{}

\newtheorem{theorem}{Theorem}[section]
\newtheorem{corollary}{Corollary}[section]

\newtheorem{definition}{Definition}[section]

\usepackage{lipsum}
\usepackage{enumerate}
\usepackage{amssymb}
\usepackage{amsmath}

\numberwithin{equation}{section}

\begin{document}

\title{Generalization of Apollonius Circle}


%    author one information
\author{Ömer Avcı\footnote{Department of Electrical $\&$ Electronics Engineering, Bo\u{g}azi\c{c}i University, 34342, Bebek, Istanbul, Turkey} \hspace{0.2in} Ömer Talip Akalın \footnote{Department of Computer Engineering, Bo\u{g}azi\c{c}i University, 34342, Bebek, Istanbul, Turkey} \hspace{0.2in} Faruk Avcı \footnote{Department of Computer Engineering, Istanbul Technical University, 34485, Maslak, Istanbul, Turkey}  \hspace{0.2in} Halil Salih Orhan\footnote{Department of Computer Engineering, Bo\u{g}azi\c{c}i University, 34342, Bebek, Istanbul, Turkey}  }

\maketitle
\date{}
\begin{abstract}
Apollonius of Perga, showed that for two given points $A,B$ in the Euclidean plane and a positive real number $k\neq 1$, geometric locus of the points $X$ that satisfies the equation $|XA|=k|XB|$ is a circle. This circle is called Apollonius circle. In this paper we generalize the definition of the Apollonius circle for two given circles $\Gamma_1,\Gamma_2$ and we show that geometric locus of the points $X$ with the ratio of the power with respect to the circles $\Gamma_1,\Gamma_2$ is constant, is also a circle. Using this we generalize the definition of Apollonius Circle, and generalize some results about Apollonius Circle.
%Later in the paper we generalize the definition for higher dimensional Euclidean space.

\end{abstract}

\section{Preliminaries}
%In this section we will state the definitions we use and some useful, well-known theorems.

\begin{theorem}(Apollonius Theorem)
For points $A,B$ in the Euclidean plane, and a positive real number $k\neq 1$, the points $X$ which satisfies the equation
$$|XA|=k|BX|$$
forms a circle. 
When $k=1$, they form the line perpendicular to $AB$ at the middle point of $[AB]$ \cite{prasolov}.

\end{theorem}
\begin{figure}[ht]
\includegraphics[width=10cm]{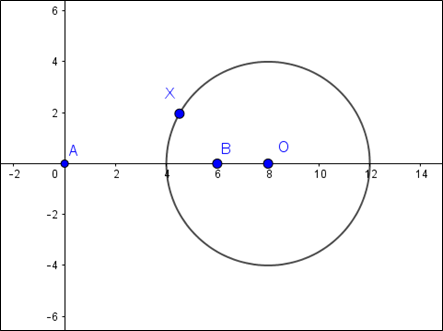}
\centering
\caption{An example of the Apollonius Thorem for fixed points $A,B$ and $k=2$}

\end{figure}

\begin{definition} \label{apolldef}
For three different points $A,B,C$ in Euclidean plane such that $A$ is not on the perpendicular
bisector of the segment $[BC]$. We will use the notation $K_A(B,C)$ for the circle which consists
of points that holds the equation 
$$\frac{|XB|}{|XC|}=\frac{|AB|}{|AC|}$$
and $M_A(B,C)$ for the center and $r_A(B,C)$ for the radius of that circle. During
this article, we will call the notation $K_A(B,C)$, \textit{Apollonius Circle} of point $A$ to the points $B,C$.

\end{definition}
\begin{figure}[!ht]
\includegraphics[width=10cm]{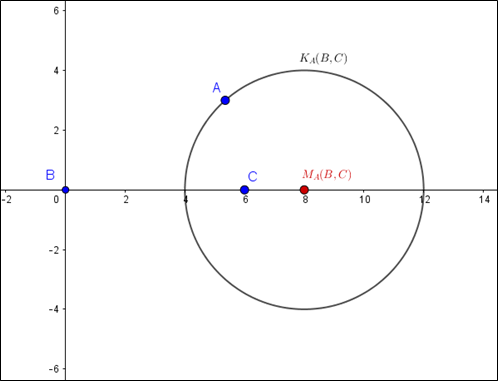}
\centering
\caption{A brief explanation of the Definiton \ref{apolldef}}

\end{figure}

\begin{definition}
For a point $A$ and a circle $\Gamma$ with center $O$ and radius $r$ in the Euclidean plane. The power of the point A with respect to circle $\Gamma$ is shown by $P_\Gamma(A)$ and it is equal to:
$$P_\Gamma(A)=|OA|^2-r^2$$
%We can use this definition for $n-1$-spheres similarly.
\begin{figure}[!ht]
\includegraphics[width=10cm]{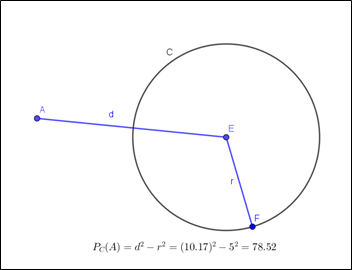}
\centering
\caption{Power of point A to the circle C}

\end{figure}

\end{definition}

\begin{definition}
For two circles $\Gamma_1,\Gamma_2$ with the centers $O_1,O_2$ respectively, in the Euclidean space the points with equal
power to the circles i.e. the points $X$ satisfying the equation
$$P_{\Gamma_1}(X)=P_{\Gamma_2}(X)$$
forms a line which is perpendicular to the line $O_1O_2$, and it is called the radical axis of the circles $\Gamma_1,\Gamma_2$ \cite{coxeter}.
\end{definition}
\begin{figure}[!ht]
\includegraphics[width=10cm]{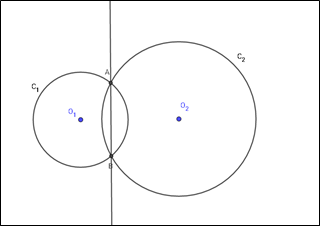}
\centering
\caption{Radical Axis of the circles $C_1,C_2$}

\end{figure}

\begin{theorem} \label{22}
For a given scalene triangle $\Delta ABC$, the line $BC$ and the bisectors of the angle $\angle BAC$ intersects at two points $X,Y$, and those points lies on the $K_A(B,C)$. 
\cite{altshiller}
\begin{figure}[!ht]
\includegraphics[width=10cm]{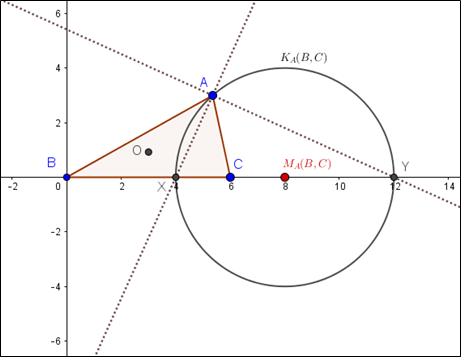}
\centering
\caption{Graphical explanation of Theorem \ref{22}}
\end{figure}
\end{theorem}

\begin{theorem} \label{23}
Let $O$ be the circumcenter of the a scalene triangle $\Delta ABC$ then, the line $AO$ is a tangent to the
circle $K_A(B,C)$ which also means that $AO \perp A M_A(B,C)$.
\begin{figure}[!ht]
\includegraphics[width=10cm]{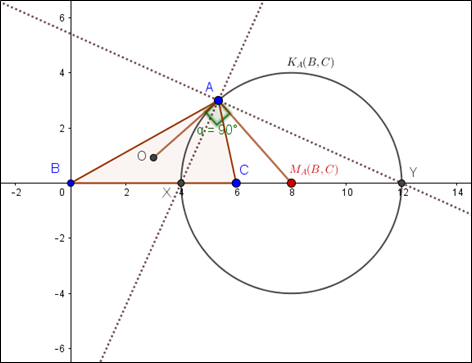}
\centering
\caption{Graphical explanation of Theorem \ref{23}}
\end{figure}
\end{theorem}

\begin{theorem} \label{24}
Let $\Delta ABC$ be a scalane triangle. Then the points $M_A(B,C),M_B(C,A),M_C(A,B)$ are collinear. This line is called Lemoine line \cite{lemoine}. Also there exists two points $S_1,S_2$ and they lie on all the circles $K_A(B,C),K_B(C,A),K_C(A,B)$. If $O$ is the circumcenter of the triangle $\Delta ABC$, then the points $S_1,S_2,O$ are collinear.
\begin{figure}[!ht]
\includegraphics[width=10cm]{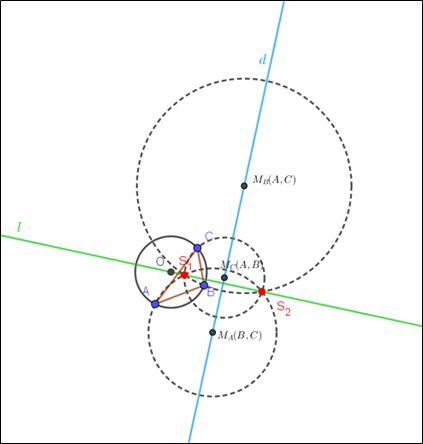}
\centering
\caption{Graphical explanation of Theorem \ref{24}}
\end{figure}
\end{theorem}

\section{Results }

%In this section, we present our findings.

\begin{theorem}(Generalized Apollonius Circle Theorem) \label{maint}
Let $\Gamma_1,\Gamma_2$ be two circles given in the plane. Let $k$ be a real number. Then the geometric locus of the points $X$ that satisfies the equation 
\begin{equation} \label{apollo}
    P_{\Gamma_1}(X)=k P_{\Gamma_2}(X) 
\end{equation}
is a circle (possibly degenerate - line, single point or non-existent)
\end{theorem}

\begin{proof}
Let $O_1,O_2$ be the centers and $r_1,r_2$ be the radii of the circles $\Gamma_1,\Gamma_2$ respectively.
Without loss of generality let us choose our coordinate system such that:
$$O_1=(0,0), \; O_2=(|O_1O_2|,0)$$
Let $X=(a,b)$ be a point that satisfies the desired relationship. Then we have:
$$(a^2+b^2-r_1^2)=k((a-|O_1O_2|)^2+b^2-r_2^2)$$
we can write it down as:
$$(k-1)a^2-2ak|O_1O_2|+(k-1)b^2=-k|O_1O_2|^2+kr_2^2-r_1^2$$
If $k=1$ this equations yields:
$$-2a=-|O_1O_2|^2+r_2^2-r_1^2$$
from this equation we can see that $a$ is constant but $b$ can be anything therefore the locus points is a line that is perpendicular to $O_1O_2$ and it is known as the radical axis.
If $k\neq 1$ then we can manipulate the equation we have into a circle equation such that:
\begin{equation} \label{apoeq}
    \Big(a-\frac{k}{k-1}|O_1O| \Big)^2+b^2=\frac{kr_2^2-r_1^2}{k-1}+\frac{k}{(k-1)^2}|O_1O_2|^2
\end{equation}
Now we can see that it is a circle equation with center $M=(\frac{k}{k-1}|O_1O|,0)$ we can see that depending on the value of the right hand side the geometric locus will be a circle($RHS>0$) or a single point ($RHS=0$) or non-existent ($RHS<0$).
We leave the calculation of when the right hand side of the equation is non-negative for later. Note that if the RHS of the equation is positive then the radius of the circle will be:
\begin{equation} \label{yaricap}
  r^2 = \frac{kr_2^2-r_1^2}{k-1} + \frac{k}{(k-1)^2}|O_1O_2|^2 
\end{equation}
Note that we can observe that the center $M$ lies on the line $O_1O_2$ and it satisfies the relationship:
\begin{equation} \label{oran}
    \frac{\overrightarrow{O_1M}}{\overrightarrow{O_2M}}=k=\frac{P_{\Gamma_1}(X)}{P_{\Gamma_2}(X)}
\end{equation}

\end{proof}

\begin{theorem}(Possible values of the power ratio)
Let $\Gamma_1,\Gamma_2$ be two intersecting circles given in the plane. Let $k$ be a real number. Then the geometric locus of the points $X$ that satisfies the equation 
$$P_{\Gamma_1}(X)=k P_{\Gamma_2}(X)$$ is
a circle for the all $k \in \mathbb{R}-\{1\}$.

\end{theorem}

\begin{proof}
Let the intersection points of the $\Gamma_1,\Gamma_2$ be $A,B$. Then we have:
$$P_{\Gamma_1}(A)=P_{\Gamma_2}(A)=P_{\Gamma_1}(B)=P_{\Gamma_2}(B)=0$$
We have shown that the equation \ref{apoeq} is a circle equation. We also know that it is not a line when $k\neq 1$. We can clearly see that the points $A,B$ satisfies the the equation \ref{apollo}, therefore there is at least two points satisfying the circle equation \ref{apoeq} which means that it can not be a single point or non-existent therefore it must be a circle.
\end{proof}

\begin{theorem}(Possible values of the power ratio)
Let $\Gamma_1,\Gamma_2$ be two circles that are not intersecting at two different points. Let $k$ be a real number. Then the geometric locus of the points $X$ that satisfies the equation 
$$P_{\Gamma_1}(X)=k P_{\Gamma_2}(X)$$ is:
\begin{enumerate}[(i)]
    \item A line when $k=1$
    \item A single point if $$k=\frac{r_1^2+r_2^2-|O_1O_2|^2 \pm \sqrt{(r_1^2+r_2^2-|O_1O_2|^2)^2-4r_1^2r_2^2}}{2r_2^2}$$
    \item A circle if 
    $$\frac{r_1^2+r_2^2-|O_1O_2|^2 + \sqrt{(r_1^2+r_2^2-|O_1O_2|^2)^2-4r_1^2r_2^2}}{2r_2^2}<k$$
    or 
    $$k<\frac{r_1^2+r_2^2-|O_1O_2|^2 - \sqrt{(r_1^2+r_2^2-|O_1O_2|^2)^2-4r_1^2r_2^2}}{2r_2^2}$$
    \item Non-existent otherwise
\end{enumerate}

\end{theorem}

\begin{proof}
In the equation \ref{apoeq} we stated that the shape of the geometric locus is dependent of whether the right hand side is positive, zero or negative. We stated that if it is positive the locus is a circle, if it is zero locus consists of only a single point and if it is negative there is no point on the locus. So let us find when it is positive or negative. If we make equal the denominators they become $(k-1)^2$ which is positive so we need to control nominator and it is equal to:
$$(k-1)(kr_2^2-r_1^2)+k|O_1O_2|^2=k^2r_2^2+k(|O_1O_2|^2-r_1^2-r_2^2)+r_1^2$$
We can see that it is a quadratic equation($k$ as a variable and $r_1,r_2$ as constants) with positive leading coefficient. Then the equation is negative between the roots and positive otherwise. So let us find the roots of the equation. We find the roots as:
$$k_{1,2}=\frac{r_1^2+r_2^2-|O_1O_2|^2 \pm \sqrt{(r_1^2+r_2^2-|O_1O_2|^2)^2-4r_1^2r_2^2}}{2r_2^2}$$
We can show that real roots exists because the discriminant is positive. This is because we assume that the circles $\Gamma_1,\Gamma_2$ does not intersect at two distinct points therefore we either have $|O_1O_2|\geq r_1+r_2$ or $|O_1O_2|\leq |r_1-r_2|$. Then we have:
$$(r_1+r_2)^2 \geq |O_1O_2|^2 \Rightarrow r_1^2+r_2^2+2r_1r_2-|O_1O_2|^2 \geq 0$$
and
$$(r_1-r_2)^2 \geq |O_1O_2|^2 \Rightarrow r_1^2+r_2^2-2r_1r_2-|O_1O_2|^2 \geq 0$$
therefore
$$(r_1^2+r_2^2-|O_1O_2|^2)^2-4r_1^2r_2^2 \geq 0$$
so we are done.
\end{proof}

\begin{definition}
Here we will generalize the Definition \ref{apolldef}. Let $A$ be a point and $\Gamma_1,\Gamma_2$ be two circles in the plane. Then we know that the geometric locus of the points $X$ satisfying the equation:
$$\frac{P_{\Gamma_1}(X)}{P_{\Gamma_2}(X)}=\frac{P_{\Gamma_1}(A)}{P_{\Gamma_2}(A)}$$
is a circle(possibly degenerate). We will show this circle with $K_{A}(\Gamma_1,\Gamma_2)$ and we will call it Apollonius circle of point $A$ to the circles $\Gamma_1,\Gamma_2$. Also let us denote the center of this circle as $M_{A}(\Gamma_1,\Gamma_2)$ and the radius of the circle as $r_{A}(\Gamma_1,\Gamma_2)$.
\begin{figure}[!ht]
\includegraphics[width=10cm]{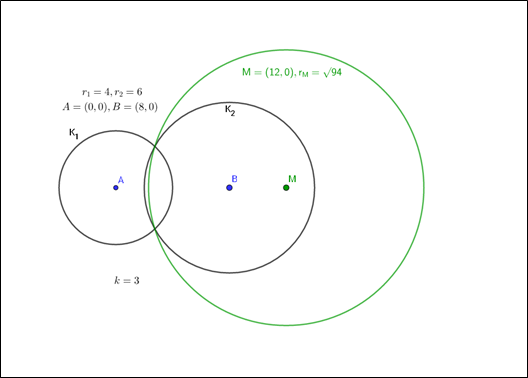}
\centering
\caption{An example of generalized Apollonius circle}
\end{figure}

\end{definition}

\begin{corollary}
If there is a point $S$ that lies on both $\Gamma_1$ and $\Gamma_2$ then $S$ also lies on the $K_A(\Gamma_1,\Gamma_2)$.
\end{corollary}

\begin{corollary}
If the circles $\Gamma_1,\Gamma_2$ intersects at the points $S_1,S_2$, then $K_A(\Gamma_1,\Gamma_2)$ is equivalent to the circumcircle of the triangle $\Delta S_1AS_2$.
\end{corollary}

\begin{theorem} \label{cool}
Let $\Gamma_1,\Gamma_2,\Gamma_3$ be three circles in the plane with the centers $O_1,O_2,O_3$ and radii $r_1,r_2,r_3$ respectively. Then the points $M_{O_1}(\Gamma_2,\Gamma_3)$, $M_{O_2}(\Gamma_3,\Gamma_1)$ and $M_{O_3}(\Gamma_1,\Gamma_2)$ are collinear if and only if the following equation holds:
\begin{equation} \label{ozeq}
    \sum_{i=1}^3|O_iO_{i+1}|^2|O_i O_{i+2}|^2(r_{i+1}^2-r_{i+2}^2)=\sum_{i=1}^3 r_{i+1}^2r_{i+2}^2(|O_iO_{i+1}|^2-|O_i O_{i+2}|^2)
\end{equation}
where $O_{i+3}=O_i$ and $r_{i+3}=r_i$ for $i=1,2,3$.
\end{theorem}

\begin{proof}
In this theorem we present a generalization of the Lemoine line for the generalized definition of the Apollonius circle.
For convenience let us show $M_{O_i}(\Gamma_{i+1},\Gamma_{i+2})$ with $M_i$ for $i=1,2,3$ (note that $\Gamma_{i+3}=\Gamma_i$). In the proof of  \ref{maint} we have shown that $M_i$ lies on the line $O_{i+1}O_{i+2}$ and from the Equation \ref{oran} we have:
\begin{equation} 
    \frac{\overrightarrow{O_{i+1}M_i}}{\overrightarrow{O_{i+2}M_i}}=\frac{P_{\Gamma_{i+1}}(O_i)}{P_{\Gamma_{i+2}}(O_i)}
\end{equation}
Now from Menelaus Theorem we know that, $M_1,M_2,M_3$ are collinear if and only if the following equation holds:
$$\frac{\overrightarrow{O_{2}M_1}}{\overrightarrow{O_{3}M_1}} \frac{\overrightarrow{O_{3}M_2}}{\overrightarrow{O_{1}M_2}} \frac{\overrightarrow{O_{1}M_3}}{\overrightarrow{O_{2}M_3}}=1$$
From the Equation \ref{oran} we can transform this equation into this:
\begin{equation} \label{oran2}
 \frac{P_{\Gamma_{2}}(O_1)}{P_{\Gamma_{3}}(O_1)}\frac{P_{\Gamma_{3}}(O_2)}{P_{\Gamma_{1}}(O_2)}\frac{P_{\Gamma_{1}}(O_3)}{P_{\Gamma_{2}}(O_3)}  =1 
\end{equation}
and we also know that $P_{\Gamma_j}(O_i)=|O_iO_j|^2-r_j^2$ therefore the equation transform into:
\begin{equation}
\frac{\Big(|O_1O_2|^2-r_2^2\Big)}{\Big(|O_1O_3|^2-r_3^2\Big)} \frac{\Big(|O_2O_3|^2-r_3^2\Big)}{\Big(|O_2O_1|^2-r_1^2\Big)}
\frac{\Big(|O_3O_1|^2-r_1^2\Big)}{\Big(|O_3O_2|^2-r_2^2\Big)}=1    
\end{equation}
and if we rearrange this, we get the desired equation.

\end{proof}

\begin{corollary}
Let $\Gamma_1,\Gamma_2,\Gamma_3$ be three circles in the plane with the centers $O_1,O_2,O_3$ and radii $r_1,r_2,r_3$ respectively. If the centers $O_1,O_2,O_3$ forms an equilateral triangle, or all three of the radii of the circles $\Gamma_1,\Gamma_2,\Gamma_3$ are equal then the points $M_{O_1}(\Gamma_2,\Gamma_3)$, $M_{O_2}(\Gamma_3,\Gamma_1)$ and $M_{O_3}(\Gamma_1,\Gamma_2)$ are collinear.
\end{corollary}
\begin{proof}
It is clear that if we have 
$$|O_1O_2|=|O_2O_3|=|O_3O_1|$$
or
$$r_1=r_2=r_3$$
both sides of the equation \ref{ozeq} is equal to zero and therefore the equation holds.
\end{proof}

\begin{figure}[!ht]
\includegraphics[width=10cm]{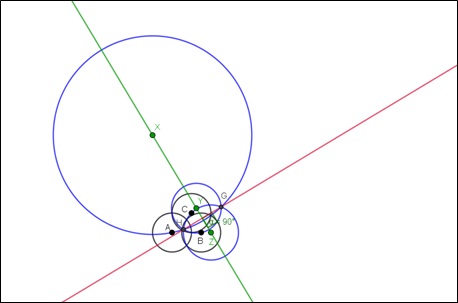}
\centering
\caption{Theorem \ref{cool} for three equal-radius circles centered at $A,B,C$}
\end{figure}

\begin{figure}[!ht]
\includegraphics[width=10cm]{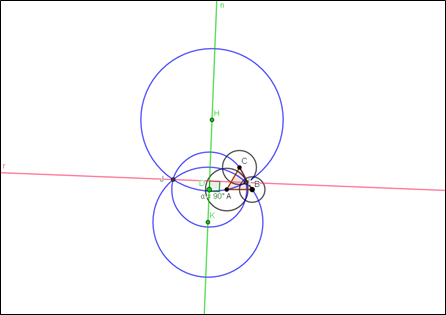}
\centering
\caption{Theorem \ref{cool} when centers of the circles form an equilateral triangle}
\end{figure}
\begin{theorem}
Let $\Gamma_1,\Gamma_2,\Gamma_3$ be three circles in the plane with the centers $O_1,O_2,O_3$ and radii $r_1,r_2,r_3$ respectively. Let $O$ be the circumcenter of the triangle $\Delta O_1O_2O_3$.
Assume that the points $M_{O_1}(\Gamma_2,\Gamma_3)$, $M_{O_2}(\Gamma_3,\Gamma_1)$ and $M_{O_3}(\Gamma_1,\Gamma_2)$ are collinear. Let the line $l_i$ be the radical axis of the circles $K_{O_{i+1}}(\Gamma_{i+2},\Gamma_{i+3})$ and $K_{O_{i+2}}(\Gamma_{i+3},\Gamma_{i+1})$ for $i=1,2,3$. Then the lines $l_1,l_2,l_3$ are the same line and this line passes through the point $O$.
\end{theorem}

\begin{proof}
For convenience let us show the center $M_{O_i}(\Gamma_{i+1},\Gamma_{i+2})$ with $M_i$ and the circles $K_{O_i}(\Gamma_{i+1},\Gamma_{i+2})$ with $K_i$ and $r_{O_i}(\Gamma_{i+1},\Gamma_{i+2})$ with $R_i$ for $i=1,2,3$.
We know that radical axis of two circles is perpendicular to the line passing through the centers of the circles. And since $M_1,M_2,M_3$ are collinear then all of the lines $l_1,l_2,l_3$ are perpendicular to the line passing through $M_1,M_2,M_3$. Therefore $l_1 \parallel l_2 \parallel l_3 $ which means that they coincide or never intersect. If we show that all of this three lines pass through $O$ then we also show that they are the same line. It is sufficient to show that $O\in l_1$. From the definition of radical axis we know that:
$$O\in l_1 \Leftrightarrow P_{K_2}(O)=P_{K_3}(O)$$
Let $N_1,N_2,N_3$ be the middle points of the $[O_2O_3],[O_3O_1],[O_1O_2]$ respectively. Also let $r$ be the radius of the circle $(O_1O_2O_3)$. We know that $M_2\in O_1O_3$ and $ON_2 \perp O_1O_3$, then from Pythagoras Theorem we have:
$$|OM_2|^2=|M_2N_2|^2+|ON_2|^2$$
and also we have:
$$r^2=|OO_1|^2=|O_1N_2|^2+|ON_2|^2$$
therefore:
$$|ON_2|^2=r^2-\frac{1}{4}|O_1O_3|^2$$
so we have:
$$|OM_2|^2=|M_2N_2|^2+r^2-\frac{1}{4}|O_1O_3|^2$$
then we will calculate $|M_2N_2|$, let us say:
$$ \frac{\overrightarrow{O_{1}M_2}}{\overrightarrow{O_{3}M_2}}=k $$
From Equation \ref{oran} we know that:
$$k=\frac{|O_2O_1|^2-r_1^2}{|O_2O_3|^2-r_3^2}$$
Since $N_2$ is the middle point of $O_1O_3$, we have:
$$ \frac{\overrightarrow{O_{1}N_2}}{\overrightarrow{O_{3}N_2}}=-1$$
By combining them we get:
$$\frac{\overrightarrow{M_{2}N_2}}{\overrightarrow{O_{3}O_1}}=\frac{k+1}{2(k-1)} $$
therefore we have:
$$|M_2N_2|^2=\Big(\frac{k+1}{2(k-1)} \Big)^2$$
and then we get:
$$|OM_2|^2=|M_2N_2|^2-|O_1N_2|^2+|OO_1|^2 =|O_1O_3|^2 \frac{k}{(k-1)^2}+r^2$$
From the definition of power we have:
$$P_{K_2}(O)=|OM_2|^2-R_2^2$$
From Equation \ref{yaricap} we know that:
$$R_2^2= \frac{k}{(k-1)^2}|O_1O_3|^2+\frac{kr_3^2-r_1^2}{k-1}$$
So by combining them we get:
$$P_{K_2}(O)=r^2+\frac{kr_3^2-r_1^2}{k-1}$$
and if we put $k=\frac{|O_2O_1|^2-r_1^2}{|O_2O_3|^2-r_3^2}$, we finally get:
$$P_{K_2}(O)=r^2+\frac{(|O_1O_2|^2-r_1^2)r_3^2-(|O_2O_3|^2-r_3^2)r_1^2}{|O_1O_2|^2+r_3^2-|O_2O_3|^2-r_1^2}$$
and then we can see that $P_{K_2}(O)=P_{K_3}(O)$ if and only if:
$$\frac{(|O_1O_2|^2-r_1^2)r_3^2-(|O_2O_3|^2-r_3^2)r_1^2}{|O_1O_2|^2+r_3^2-|O_2O_3|^2-r_1^2}=\frac{(|O_1O_3|^2-r_1^2)r_2^2-(|O_2O_3|^2-r_2^2)r_1^2}{|O_1O_3|^2+r_2^2-|O_2O_3|^2-r_1^2}$$
and if we rearrange the equation we can see that it is equivalent to the equation \ref{ozeq}.

\end{proof}


\bibliographystyle{amsplain}

\end{document}